\documentclass[11pt,a4paper]{article}
\usepackage{epsf,epsfig,amsfonts,amsgen,amsmath,amstext,amsbsy,amsopn,amsthm,lineno}
\usepackage{color}
\usepackage{graphicx}
\newtheorem{theorem}{Theorem}

\newtheorem{lemma}{Lemma}

\theoremstyle{definition}

\newtheorem{claim}{Claim}

\newtheorem{subclaim}{Claim}[claim]

\newtheorem{remark}{Remark}

\newtheorem{case}{Case}
\newtheorem{subcase}{Subcase}[case]

\begin{document}
\title
{\bf\Large Hamilton cycles in almost distance-hereditary graphs}

\author{
Bing Chen$^a$\thanks{Supported by NSFC (No.~11271300) and
the Scientific Research Program of Shaanxi Provincial Education
Department (No. 2013JK0580). E-mail address: cbing\_2004@163.com (B. Chen).}
and Bo Ning$^{b}$\thanks{Corresponding author. Supported by NSFC (No.~11271300)
and the Doctorate Foundation of Northwestern Polytechnical University
(No. cx201326). E-mail address: ningbo\_math84@mail.nwpu.edu.cn (B. Ning).}\\
\small $^{a}$Department of Applied Mathematics, School of Science, \\
\small Xi'an University of Technology, Xi'an, Shaanxi 710048, P.R.~China\\
\small $^{b}$Department of Applied Mathematics, School of Science,\\
\small  Northwestern Polytechnical University, Xi'an, Shaanxi 710072,
P.R.~China\\[2mm]}
\date{}
\maketitle

\begin{abstract} Let $G$ be a graph on $n\geq 3$ vertices. A
graph $G$ is almost distance-hereditary if each connected
induced subgraph $H$ of $G$ has the property $d_{H}(x,y)\leq d_{G}(x,y)+1$
for any pair of vertices $x,y\in V(H)$. A graph $G$ is called
1-heavy (2-heavy) if at least one (two) of the end vertices of
each induced subgraph of $G$ isomorphic to $K_{1,3}$ (a claw)
has (have) degree at least $n/2$, and called claw-heavy if each
claw of $G$ has a pair of end vertices with degree sum at least $n$.
Thus every 2-heavy graph is claw-heavy. In this paper we prove
the following two results: (1) Every 2-connected, claw-heavy and
almost distance-hereditary graph is Hamiltonian. (2) Every 3-connected,
1-heavy and almost distance-hereditary graph is Hamiltonian.
In particular, the first result improves a previous theorem of
Feng and Guo. Both results are sharp in some sense.
\medskip

\noindent {\bf Keywords:} Hamilton cycle; Almost distance-hereditary graph;
Claw-free graph; 1-heavy graph; 2-heavy graph; Claw-heavy graph
\smallskip

\noindent {\bf AMS Subject Classification (2000):} 05C38, 05C45
\end{abstract}

\section{Introduction}

We use Bondy and Murty \cite{Bondy_Murty} for terminology and
notation not defined here and consider simple graphs only.

Let $G$ be a graph. For a vertex $v$ and a subgraph $H$ of
$G$, we use $N_H(v)$ to denote the set, and $d_H(v)$ the
number, of neighbors of $v$ in $H$, respectively. We call
$d_H(v)$ the \emph{degree} of $v$ in $H$. For $x,y\in V(G)$,
an $(x,y)$-\emph{path} is a path connecting $x$ and $y$. If
$x,y\in V(H)$, the \emph{distance} between $x$ and $y$ in $H$,
denoted by $d_H(x,y)$, is the length of a shortest $(x,y)$-path
in $H$. When there is no danger of ambiguity, $N_G(v)$, $d_G(v)$
and $d_G(x,y)$ are abbreviated to $N(v)$, $d(v)$ and $d(x,y)$,
respectively.

A graph is called \emph{Hamiltonian} if it contains a
Hamilton cycle, i.e., a cycle passing through all the
vertices of the graph. The study of cycles, especially
Hamilton cycles, maybe one of the most important and most
studied areas of graph theory. It is well-known that to
determine whether a given graph contains a Hamilton cycle
is $\mathcal{NP}$-complete, shown by R.M. Karp \cite{Karp}.
However, if we only consider some restricted graph classes,
then the situation is completely changed. A graph $G$ is
called \emph{distance-hereditary} if each connected induced
subgraph $H$ has the property that $d_{H}(x,y)=d_{G}(x,y)$
for any pair of vertices $x,y$ in $H$. This concept was
introduced by Howorka \cite{Howorka} and a complete
characterization of distance-hereditary graphs can be
found in \cite{Howorka}. In 2002, Hsieh, Ho, Hsu and
Ko \cite{Hsieh_Ho_Hsu_Ko} obtained an $O(|V|+|E|)$-time
algorithm to solve the Hamiltonian problem on distance-hereditary
graphs. Some other optimization problems can also be solved
in linear time for distance-hereditary graphs although
they are proved to be $\mathcal{NP}$-hard for more general graphs.
For references in this direction, we refer to
\cite{Alessandro_Marina,Cogis_Thierry}.

A graph $G$ is called \emph{almost distance-hereditary}
if each connected induced subgraph $H$ of $G$ has the
property $d_{H}(x,y)\leq d_{G}(x,y)+1$ for any pair of
vertices $x,y\in V(H)$. For some properties and a
characterization of almost-distance hereditary graphs,
we refer to \cite{Aider}.

Let $G$ be a graph. An induced subgraph of $G$
isomorphic to $K_{1,3}$ is called a \emph{claw},
the vertex of degree 3 in it is called its \emph{center}
and the other vertices are its \emph{end vertices}.
$G$ is called \emph{claw-free} if $G$ contains no claw.
Throughout this paper, whenever the vertices of a
claw are listed, its center is always the first one.

The class of claw-free graphs is important in graph
theory. Maybe one big reason is due to Matthews and Sumner's
conjecture \cite{Mattews_Sumner} which states that every
4-connected claw-free graph is Hamiltonian. Many results
about the existence of Hamilton cycles in claw-free graphs
have been obtained. For surveys on Matthews and Sumner's
conjecture and on claw-free graphs, we refer the reader to
\cite{Broersma_Ryjacek_Vrana} and \cite{Faudree_Flandrin_Ryjacek},
respectively.

In particular, Feng and Guo \cite{Feng_Guo_1} gave the following
result on Hamiltonicity of almost distance-hereditary claw-free graphs.

\begin{theorem}[Feng and Guo \cite{Feng_Guo_1}]\label{th1}
Let $G$ be a 2-connected claw-free graph. If $G$ is almost
distance-hereditary, then $G$ is Hamiltonian.
\end{theorem}

Let $G$ be a graph. A vertex $v$ of $G$ on $n$ vertices is called
\emph{heavy} if $d(v)\geq n/2$. Broersma et al.
\cite{Broersma_Ryjacek_Schiermeyer} introduced the concepts of
1-heavy graph and 2-heavy graph, and the concept of claw-heavy
graph was introduced by Fujisawa \cite{Fujisawa}. Following
\cite{Broersma_Ryjacek_Schiermeyer,Fujisawa}, we say that a claw in
$G$ is \emph{1-heavy} (\emph{2-heavy}) if at least one (two) of its
end vertices is (are) heavy. $G$ is called 1-heavy
(\emph{2-heavy}) if every claw of it is 1-heavy (2-heavy),
and called \emph{claw-heavy} if every claw of it has two
end vertices with degree sum at least $n$. It is easily
seen that every claw-free graph is 1-heavy (2-heavy,
claw-heavy), every 2-heavy graph is claw-heavy but
not every claw-heavy graph is 2-heavy, and every
claw-heavy graph is 1-heavy but not every 1-heavy
graph is claw-heavy.

Broersma et al. \cite{Broersma_Ryjacek_Schiermeyer} proved some
sufficient conditions for Hamiltonicity of 1-heavy graphs and
2-heavy graphs. Motivated by the works of Broersma et al.
\cite{Broersma_Ryjacek_Schiermeyer}, Feng and Guo
\cite{Feng_Guo_2} extended Theorem \ref{th1} to a larger
graph class of 2-heavy graphs.

\begin{theorem}[Feng and Guo \cite{Feng_Guo_2}]\label{th2}
Let $G$ be a 2-connected 2-heavy graph. If $G$ is almost
distance-hereditary, then $G$ is Hamiltonian.
\end{theorem}

Later, Chen et al. \cite{Chen_Zhang_Qiao} extended two results
of Broersma et al. \cite{Broersma_Ryjacek_Schiermeyer} to the
more larger graph class of claw-heavy graphs. Motivated by
Chen et al.' previous works \cite{Chen_Zhang_Qiao}, in
this paper we obtain the following two theorems which
extend Theorem \ref{th1} and Theorem \ref{th2}.
In particular, Theorem \ref{th3} improves Theorem \ref{th2}.

\begin{theorem}\label{th3}
Let $G$ be a 2-connected claw-heavy graph. If $G$ is almost
distance-hereditary, then $G$ is Hamiltonian.
\end{theorem}

\begin{theorem}\label{th4}
Let $G$ be a 3-connected 1-heavy graph. If $G$ is almost
distance-hereditary, then $G$ is Hamiltonian.
\end{theorem}

\begin{remark}
The graph in Fig.1 shows that the result in Theorem \ref{th3} indeed
strengthen that in Theorem \ref{th2}. As shown in
\cite[Fig.2]{Chen_Zhang_Qiao}, let $n\geq 10$ be an even
integer and $K_{n/2}+K_{n/2-3}$ denote the join of two
complete graphs $K_{n/2}$ and $K_{n/2-3}$.
Choose a vertex $y\in V(K_{n/2})$ and
construct a graph $G$ with
$V(G)=V(K_{n/2}+K_{n/2-3})\cup \{v,u,x\}$ and $E(G)=E(K_{n/2}+
K_{n/2-3})\cup \{uv,uy,ux\}\cup \{vw,xw: w\in V(K_{n/2-3})\}$.
It is easy to see that $G$ is a Hamiltonian graph satisfying
the condition of Theorem \ref{th3}, but not the condition
of Theorem \ref{th2}.

\vskip -5.5cm \hspace
{-2.6cm}\includegraphics[scale=0.6]{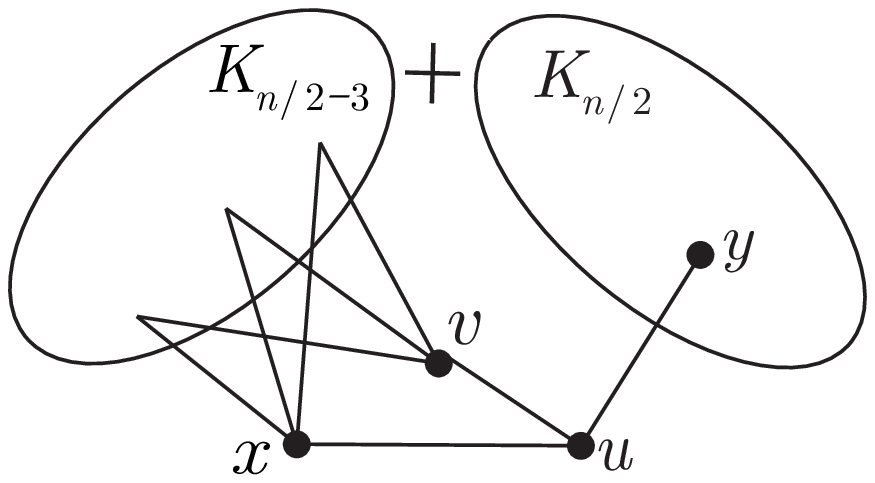}\vskip -0.5cm
 \hspace {6cm}
 {\small{Fig.1}}
\vskip 0.5cm

\begin{remark}
In \cite{Feng_Guo_2}, Feng and Guo constructed a graph which
is 2-connected almost distance-hereditary 1-heavy graph but
not Hamiltonian.  Hence the connectivity condition in our
Theorem \ref{th4} is sharp.
\end{remark}

We postpone the proofs of Theorem \ref{th3} and \ref{th4} to the next section.
\end{remark}

\section{Proofs of Theorems 3 and 4}
Before the proofs, we first introduce some additional
terminology and notation.

Let $H$ be a path (a cycle) with a given orientation.
We denote $\overleftarrow{H}$ by the same graph as
$H$ but with the reverse orientation. When $v\in V(H)$,
$v_H^+$ and $v_H^-$ denote the successor (if it exists) and
the predecessor (if it exists) of $v$ on $H$, respectively.
If $S\subseteq V(H)$, then define $S^+_H=\{s^+_H: s\in S\}$ and
$S^-_H=\{s^-_H: s\in S\}$. If there is no danger of ambiguity,
then we denote $v_H^+$, $v_H^-$, $S_H^+$ and $S_H^-$ by
$v^+$, $v^-$,  $S^+$ and $S^-$, respectively. For
two vertices $u,v\in V(H)$, we use $H[u,v]$ to denote the
segment of $H$ from $u$ to $v$, and by $H[u,v)$, $H(u,v]$ and $H(u,v)$,
we mean the path $H[u,v]-\{v\}$, $H[u,v]-\{u\}$ and $H[u,v]-\{u,v\}$,
respectively.

Let $G$ be a graph on $n$ vertices and $k\geq 3$ be an integer.
Recall that a vertex of degree at least $n/2$ in $G$ is a
\emph{heavy vertex}; otherwise it is \emph{light}. A claw in $G$
is called a \emph{light claw} if all its end vertices are light,
and called an \emph{o-light claw} if any pair of end vertices has
degree sum less than $n$.
A \emph{heavy cycle} in $G$ is a cycle containing all
heavy vertices in $G$. Following \cite{Li_Ryjacek_Wang_Zhang},
we use $\widetilde{E}(G)$ to denote the set $\{xy: xy\in E(G)$
or $d(x)+d(y)\geq n, x, y\in V(G)\}$. A sequence of vertices
$C=u_1u_2\ldots u_ku_1$ is called an {\it Ore-cycle} or briefly,
{\it o-cycle} of $G$, if $u_iu_{i+1}\in \widetilde{E}(G)$
for every $i\in [1,k]$, where $u_1=u_{k+1}$.

To prove Theorems \ref{th3} and \ref{th4}, the following
three lemmas are needed.

\begin{lemma}[Bollob\'{a}s and Brightwell \cite{Bollobas_Brightwell}, Shi \cite{Shi}]\label{le1}
Every 2-connected graph contains a heavy cycle.
\end{lemma}

\begin{lemma}[Li, Ryj\'a\v{c}ek, Wang and Zhang \cite{Li_Ryjacek_Wang_Zhang}]\label{le2}
Let $G$ be graph on $n$
vertices and $C'$ be an o-cycle of $G$. Then there exists a cycle $C$ of $G$ such that
$V(C')\subseteq V(C)$.
\end{lemma}

\begin{lemma}\label{le3}
Let $G$ be a non-Hamiltonian graph, $C$ be a longest heavy cycle
(a longest cycle) of $G$, $R$ a component of $G-V(C)$, and
$A=\{v_1,v_2,\ldots,v_k\}$ the set of neighbors of $R$ on $C$.
Let $u\in V(R)$, $v_i,v_j\in A$ and $P$ be a $(v_i,v_j)$-path with
all internal vertices in $R$.
Then there hold\\
$(a)$ $uv_i^-\notin\widetilde{E}(G)$, $uv_i^+\notin \widetilde{E}(G)$;\\
$(b)$ $v_i^-v_j^-\notin\widetilde{E}(G)$, $v_i^+v_j^+\notin \widetilde{E}(G)$;\\
$(c)$ If $v^-_iv_i^+\in \widetilde{E}(G)$, then $v_iv_j^-\notin \widetilde{E}(G)$,
$v_iv_j^+\notin \widetilde{E}(G)$;\\
$(d)$ Let $l\in V(C[v_i,v_j^-])\cap N(v_i)\cap N(v_j^-)$. If $v^-_iv_i^+\in
\widetilde{E}(G)$, then $l^-l^+\notin \widetilde{E}(G)$, $v_il^+\notin \widetilde{E}(G)$;\\
$(e)$ Let $l\in V(C[v_i,v_j^-])\cap N(v_i)$. If $v^-_iv_i^+\in \widetilde{E}(G)$,
then $l^-v_j^-\notin \widetilde{E}(G)$, $v_j^+l^+\notin \widetilde{E}(G)$;\\
$(f)$ Let $l\in V(C[v_i,v_j^-])$. If $v^-_il\in \widetilde{E}(G)$,
then $l^-v_j^-\notin \widetilde{E}(G)$, $l^-v_j^+\notin \widetilde{E}(G)$,
$l^+v_j^+\notin \widetilde{E}(G)$;\\
Furthermore, if $G$ is a 2-connected claw-heavy graph, then\\
$(g)$ $v_i^{-}v_i^{+}\in \widetilde{E}(G)$, $v_j^{-}v_j^{+}\in
\widetilde{E}(G)$;\\
$(h)$ $v^-_iv_i^+\in E(G)$ or $v^-_jv_j^+\in E(G)$.
\end{lemma}

\begin{proof}
$(a),(b),(c)$ The proof of $(a)$ is similar as the proof of Claim 1, and the proof
of $(b)$, $(c)$ is similar as the proof of Claim 3, in \cite[Theorem 8]{Li_Ryjacek_Wang_Zhang},
respectively.

$(d)$ Suppose that $l^-l^+\in \widetilde{E}(G)$. Then $C'=PC[v_j,v_i^-]
v_i^-v_i^+C[v_i^+,l^-]l^-l^+C[l^+,v_j^-]$ $v_j^-lv_i$ is an \emph{o}-cycle
such that $V(C)\subset V(C')$. By Lemma
\ref{le2}, there is a longer cycle $C''$ containing all vertices in $C$, contradicting the
choice of $C$. Suppose that $v_il^+\in \widetilde{E}(G)$. Then
$C'=PC[v_j,v_i^-]v_i^-v_i^+C[v_i^+,l]lv_j^-\overleftarrow{C}[v_j^-,l^+]l^+v_i$
is an \emph{o}-cycle such that $V(C)\subset V(C')$, a contradiction.

$(e)$ Suppose that $l^-v_j^-\in \widetilde{E}(G)$. Then $C'=\overleftarrow{P}
v_ilC[l,v_j^-]v_j^-l^-\overleftarrow{C}[l^-,v_i^+]v_i^+v_i^-\overleftarrow{C}
[v_i^-,v_j]$ is an \emph{o}-cycle such that $V(C)\subset V(C')$,
a contradiction. Suppose that $v_j^+l^+\in \widetilde{E}(G)$.
Then $C'=P\overleftarrow{C}[v_j,l^+]l^+v_j^+C[v_j^+,v_i^-]v_i^-v_i^+
C[v_i^+,l]lv_i$ is an \emph{o}-cycle such that $V(C)\subset V(C')$,
a contradiction.

$(f)$ Suppose that $l^-v_j^-\in \widetilde{E}(G)$. Then $C'=PC[v_j,
v_i^-]v_i^-lC[l,v_j^-]v_j^-l^-\overleftarrow{C}[l^-,v_i]$ is an
\emph{o}-cycle such that $V(C)\subset V(C')$, a contradiction.
Suppose that $l^-v_j^+\in \widetilde{E}(G)$. Then $C'=\overleftarrow{P}
C[v_i,l^-]l^-v_j^+C[v_j^+,v_i^-]v_i^-lC[l,v_j]$ is an \emph{o}-cycle
such that $V(C)\subset V(C')$, a contradiction. Suppose that $l^+v_j^+
\in \widetilde{E}(G)$. Then $C'=P\overleftarrow{C}[v_j,l^+]l^+v_j^+
C[v_j^+,v_i^-]v_i^-l\overleftarrow{C}[l,v_i]$ is an \emph{o}-cycle
such that $V(C)\subset V(C')$, a contradiction.

$(g)$, $(h)$ The proof of $(g)$ and $(h)$ is similar as the proof of Claim 2, 4 in
\cite[Theorem 8]{Li_Ryjacek_Wang_Zhang}, respectively.
\end{proof}

\noindent{}
{\bf Proof of Theorem \ref{th3}}

Let $G$ be a graph satisfying the condition of Theorem \ref{th3}.
Let $C$ be a longest cycle of $G$ and assign an orientation to
it. Suppose $G$ is not Hamiltonian. Then $V(G)\backslash V(C)\ne
\emptyset$. Let $R$ be a component of $G-C$, and $A=\{v_1,v_2,
\ldots,v_k\}$ be the set of neighbors of $R$ on $C$. Since $G$
is 2-connected, there exists a $(v_i,v_j)$-path $P=v_iu_1\ldots
u_rv_j$ with all internal vertices in $R$, and $v_i,v_j\in A$.
Choose $P$ such that:\\
$(1)$ $|V(C(v_i,v_j))|$ is as small as possible;\\
$(2)$ $|V(P)|$ is as small as possible subject to (1).

By the choice of $P$, the first claim is obvious.

\begin{claim}\label{c1}
For any two vertices $u_s,u_t\in V(P)\backslash \{v_i,v_j\}$
and $v\in V(C(v_i,v_j))$, there hold
$u_su_t\notin E(G)$ if $d_P(u_s,u_t)\geq 2$, $u_sv\notin E(G)$,
$v_iu_s\notin E(G)$ and $v_ju_s\notin E(G)$.
\end{claim}

\begin{claim}\label{c2}
There is no \emph{o}-cycle $C'$ in $G$ such that $V(C)\subset
 V(C')$.
\end{claim}

\begin{proof}
Otherwise, $C'$ is an \emph{o}-cycle such that $V(C)\subset V(C')$.
By Lemma \ref{le2}, there exists a cycle containing all vertices
in $C'$ and longer than $C$, contradicting the choice of $C$.
\end{proof}

By Lemma \ref{le3} ($h$), without loss of generality, assume
that $v^-_iv_i^+\in E(G)$.

\begin{claim}
$r=1$, that is, $V(P)=\{v_i,u_1,v_j\}$.
\end{claim}

\begin{proof}
Suppose that $r\geq 2$. Consider $H=G[V(P)\cup V(C[v_i,v_j])]-v_j$.
By Lemma \ref{le3} $(g)$, $v_i^-v_i^+\in \widetilde{E}(G)$. Since
$v_i^-v_i^+\in \widetilde{E}(G)$, $v_iv_j^-\notin E(G)$ by Lemma \ref{le3} $(c)$.
Thus $d_H(v_j^-,v_i)\geq 2$. By the choice of $P$, we have
$d_P(v_i,u_r)\geq 2$. Now by Claim \ref{c1}, $d_H(v_j^-,u_r)=
d_H(v_j^-,v_i)+d_P(v_i,u_r)\geq 4$, which yields a contradiction
to the fact $G$ is almost distance-hereditary and $d_G(v_j^-,u_r)=2$.
Hence $V(P)=\{v_i,u_1,v_j\}$.
\end{proof}

\begin{claim}\label{c4}
$|V(C[v_i,v_j])|\geq 5$.
\end{claim}

\begin{proof}
Suppose $|V(C[v_i,v_j])|=4$ or $|V(C[v_i,v_j])|=3$. This means
$C[v_i,v_j]=v_iv^+_iv^-_jv_j$ or $C[v_i,v_j]=v_iv^+_iv_j$.
Let $C'=v_iu_1v_jv_j^{-}v_j^{+}C[v_j^{+},v_i^{-}]v_i^{-}v_i^{+}
v_i$ or $C'=v_iu_1v_jv_j^{-}v_j^{+}C[v_j^{+},v_i]$. Then $C'$
is an \emph{o}-cycle such that $V(C)\subset
V(C')$, contradicting Claim \ref{c2}.
\end{proof}

Let $H=G[\{u_1,v_i^-\}\cup V(C[v_i,v_j])]-v_i$. Since $d_G(v_i^-,u_1)=2$
and $G$ is almost distance-hereditary, $d_H(v_i^-,u_1)\leq 3$.
By Lemma \ref{le3} $(c)$ and Claim \ref{c1}, we have $v_i^-v_j\notin E(G)$
and $u_1v_s\notin E(G)$, where $v_s\in C[v_i^+,v_j^-]$. It follows that
$d_H(v_i^-,u_1)= 3$ and $d_H(v_i^-,v_j)=2$. By Lemma \ref{le3} $(b)$
and $(c)$, $v_i^-v_j^-\notin E(G)$ and $v_i^+v_j\notin E(G)$.
Thus there exits a vertex $w\in C(v_i^{+},v_j^-)$ such that $v_i^-w\in E(G)$
and $wv_j\in E(G)$. Note that $w$ is well-defined by Claim \ref{c4}.

\begin{claim}\label{c5}
$v_i^-w^+\in \widetilde{E}(G)$.
\end{claim}

\begin{proof}
Suppose that $v_i^-w^+\notin \widetilde{E}(G)$. By Lemma \ref{le3} $(e)$
and symmetry, we have $v_jw^+\notin \widetilde{E}(G)$. Note that
$v_jv_i^-\notin \widetilde{E}(G)$ by Lemma \ref{le3} ($c$).
Thus $\{w,w^+,v_j,v_i^-\}$ induces an $o$-light claw in $G$,
a contradiction.
\end{proof}

\begin{claim}\label{c6}
$wv_j^+\in \widetilde{E}(G)$.
\end{claim}

\begin{proof}
First we will show that $w^-v^-_i\notin \widetilde{E}(G)$. Since
$v_j^-v_j^+\in \widetilde{E}(G)$ and $v_jw\in E(G)$, we have
$w^-v^-_i\notin \widetilde{E}(G)$ by Lemma \ref{le3} $(e)$ and
symmetry.

Next we will show that $w^-v_j\in \widetilde{E}(G)$. Suppose not.
Consider the subgraph induced by $\{w,w^-,v_j,v_i^-\}$. Note that
$v_jv_i^-\notin \widetilde{E}(G)$ by Lemma \ref{le3} ($c$) and
$w^-v^-_i\notin \widetilde{E}(G)$. Then $\{w,w^-,v_j,v_i^-\}$
induces an $o$-light claw, a contradiction.

Next we will show that $u_1w\notin \widetilde{E}(G)$. Otherwise, $C'=u_1wC[w,v_j^-]v_j^-v_j^+C[v_j^+,w^-]w^-$ $v_ju_1$ (note that $w^-v_j\in \widetilde{E}(G)$) is an \emph{o}-cycle such that $V(C)\subset V(C')$, contradicting Claim \ref{c2}.

Now we will show that $wv_j^+\in \widetilde{E}(G)$.
Consider the subgraph induced by $\{v_j,u_1,w,v_j^+\}$. Note that $u_1w\notin \widetilde{E}(G)$ by the analysis above and $u_1v_j^+\notin \widetilde{E}(G)$ by Lemma \ref{le3} ($a$). Since $G$ is claw-heavy, we have $wv_j^+\in \widetilde{E}(G)$.
\end{proof}

By Claim \ref{c5}, we have $v_i^-w^+\in \widetilde{E}(G)$.
By Lemma \ref{le3} $(f)$, we have
$wv_j^+\notin \widetilde{E}(G)$, contradicting
Claim \ref{c6}. The proof of Theorem \ref{th3} is complete. {\hfill$\Box$}

\noindent{}
{\bf Proof of Theorem \ref{th4}}

Let $G$ be a graph satisfying the condition of Theorem
\ref{th4}. By Lemma \ref{le1}, there exists a heavy cycle
in $G$. Now choose a longest heavy cycle $C$ of $G$
and assign an orientation to it. Suppose $G$ is not
Hamiltonian. Then $V(G)\backslash V(C)\ne \emptyset$.
Let $R$ be a component of $G-C$ and $A=\{w_1,w_2,\ldots,w_k\}$ be the set of
neighbors of $R$ on $C$. Since $G$ is 3-connected, for any vertex $u$
of $R$, there exists a ($u$,$C$)-fan $F$ such that $F=(u;Q_1,Q_2,Q_3)$,
where $Q_1=uu_1\ldots u_{r_1}w_i$, $Q_2=us_1\ldots s_{r_2}w_j$
and $Q_3=uy_1\ldots y_{r_3}w_k$ are three internally disjoint
paths, $V(Q_1)\cap V(C)=w_i$, $V(Q_2)\cap V(C)=w_j$,
$V(Q_3)\cap V(C)=w_k$, and $w_i,w_j,w_k$ are in the order
of the orientation of $C$.

By the choice of $C$, all internal vertices on
$F$ are not heavy. By Lemma \ref{le3} $(b)$, there is at most
one heavy vertex in $N_{C}^+(R)$ and at most one heavy
vertex in $N_{C}^-(R)$. Now assume that
$w_i^{-},w_i^{+}$ are light. Hence $w_i^{-}w_i^{+}\in E(G)$, otherwise
$\{w_i,w_i^-,w_i^+,u_{r_1}\}$ induces a
light claw, contradicting $G$ is 1-heavy.

\setcounter{claim}{0}
\begin{claim}\label{c1}
There exists a ($u$,$C$)-fan $F$ such that $V(F)=\{u,w_i,w_j,w_k\}$.
\end{claim}

\begin{proof}
Now we choose the fan $F$ in such a way that:\\
(1) $|V(Q_1)|=2$ and $w_i\in V(Q_1)$;\\
(2) $|V(C[w_i,w_j])|$ is as small as possible subject to (1);\\
(3) $|V(Q_2)|$ is as small as possible subject to (1), (2);\\
(4) $|V(C[w_k,w_i])|$ is as small as possible subject to (1), (2) and (3);\\
(5) $|V(Q_3)|$ is as small as possible subject to (1), (2), (3) and (4)

Since $G$ is 3-connected, this choice condition is well-defined. Without loss of generality, assume
$Q_1=uw_i$. Note that $w_i^-w_i^+\in E(G)$.

First we show that $V(Q_2)= \{u,w_j\}$. Suppose there exists a vertex of $V(Q_2)\backslash \{u,w_j\}$.
Without loss of generality, set $x=s_{r_2}$. Let $H=G[V(Q_1)\cup V(Q_2)\cup V(C[w_i,w_j])]-w_j$.
By Lemma \ref{le3} $(c)$, it is easy to see that $w_iw_j^-\notin E(G)$, so $d_H(w_j^-,w_i)\geq 2$,
and the choice condition (2) implies that $N_{C(w_i,w_j)}(Q_2\backslash \{u,w_j\})=\emptyset$.
This means that $d_{H}(w_i,x)=d_{F}(w_i,x)$. If $xw_i\notin E(G)$, then $d_{F}(w_i,x)\geq 2$.
Since $|V(Q_2)|>2$, we have $d_H(w_j^-,x)=d_H(w_j^-,w_i)+d_F(w_i,x)\geq 4$.
It yields a contradiction to the fact $G$ is almost distance-hereditary
and $d_G(w_j^-,x)=2$. Thus, $xw_i\in E(G)$. Let $F=(x,xw_i,xw_j,Q_2[x,u]Q_3[u,w_j])$.
Then $F$ is a $(x,C)$-fan satisfying (1), (2) and $|\{x,w_j\}|=2$, contradicting the choice
condition $(3)$, a contradiction. Hence $V(Q_2)= \{u,w_j\}$.

Next we show that $V(Q_3)=\{u,w_k\}$. Suppose there exists a vertex of $V(Q_3)\backslash
\{u,w_k\}$. Without loss of generality, set $x=y_{r_3}$. If $xw_i\notin E(G)$,
let $H=G[V(Q_1)\cup V(Q_3)\cup V(C[w_k,w_i])]-w_k$. By Lemma \ref{le3} $(c)$,
we have $w_k^+w_i\notin E(G)$. This means $d_H(w_k^+,w_i)\geq 2$.
By the choice condition (4) (5), we have $N(Q_3\backslash \{w_k\})\cap V(C(w_k,w_i))=\emptyset$.
Since $|V(Q_3)|\geq 3$, $d_H(w_i,x)\geq 2$ and we have $d_H(w_k^+,x)=d_H(w_k^+,w_i)+d_F(w_i,x)\geq 4$.
It yields a contradiction to the fact $G$ is almost distance-hereditary
and $d_G(w_k^+,x)=2$. Thus, $xw_i\in E(G)$.
Since $d_{G}(x,w_i^+)=2$ and $G$ is almost distance-hereditary, let
$H=G[V(C[w_i,w_j])\cup V(Q_3[u,x])]-\{w_i\}$, we have $d_{H}(w_i^+,x)=3$. Note that
$N_C(x)\cap C(w_i,w_j]=\emptyset$ and $N_C(u)\cap C(w_i,w_j)=\emptyset$.
It follows that $w_i^+w_j\in E(G)$, and $w_j^-w_j^+\notin \widetilde{E}(G)$ by Lemma \ref{le3} $(c)$.
Consider the subgraph induced by $\{w_j,w_i^+,w_j^+,u\}$. We know that $w_j^+$ is heavy.
Now let $H=G[\{w_i^-\}\cup V(C[w_i,w_j])\cup V(Q_3[u,x])]-\{w_i\}$.
Since $d_{G}(w_i^-,x)=2$, we have $d_{H}(w_i^-,x)=3$, and $w_i^-w_j\in E(G)$.
Consider the subgraph induced by $\{w_j,w_i^-,w_j^-,u\}$. We can see that
$w_j^-$ is heavy. Now $w_j^-w_j^+\in \widetilde{E}(G)$, a contradiction. The proof is complete.
\end{proof}
By Claim \ref{c1}, there exists a ($u,C$)-fan $F$ such that $V(F)\backslash V(C)=\{u\}$.
Suppose that $N_C(u)=\{v_1,v_2,\ldots,v_r\}$ ($r\geq 3$) and $v_1,v_2,\ldots,v_r$ are in the order
of the orientation of $C$. In the following, all subscripts of $v$ are taken modulo $r$, and
$v_0=v_r$.

\begin{claim}\label{c2}
For any vertex $v_i\in N_C(u)$ such that $v_i^-v_i^+\in \widetilde{E}(G)$, there exists
a vertex $l_i\in C[v_i^{+},v_{i+1}^-)$ such that $v_{i+1}^-l\in E(G)$
and $lv_i\in E(G)$; there exists a vertex ${l_i}'\in C(v_{i-1}^{+},v_i^-]$
such that $v_{i-1}^+s\in E(G)$ and $sv_i\in E(G)$.
\end{claim}

\begin{proof}
Let $H=G[\{u\}\cup V(C[v_i,v_{i+1}])]-v_{i+1}$. Since $d_G(v_{i+1}^-,u)=2$ and
$G$ is almost distance-hereditary, we have $d_H(v_{i+1}^-,u)\leq 3$.
By Lemma \ref{le3} $(c)$, we have $v_iv_{i+1}^-\notin E(G)$ and $d_H(v_{i+1}^-,u)= 3$.
It follows that $d_H(v_{i+1}^-,v_i)=2$. So there exists a vertex
$l_i\in C[v_i^{+},v_{i+1}^-)$ such that $v_{i+1}^-l_i\in E(G)$ and $l_iv_i\in E(G)$.
The other assertion can be proved similarly.
\end{proof}

By Lemma \ref{le3} $(b)$, there is at most
one heavy vertex in $N_{C}^+(u)$ and at most one heavy
vertex in $N_{C}^-(u)$. Since $r\geq 3$, we know that there exits
$v_j\in N_{C}(u)$, such that $v_j^-v_j^+\in E(G)$. Without loss of generality, assume that $v_1^-v_1^+\in E(G)$.
We divide the proof into two cases.

\setcounter{case}{0}

\begin{case}
$v_2^-v_2^+\notin E(G)$ and $v_r^-v_r^+\notin E(G)$.
\end{case}
Both $\{v_2,v_2^-,v_2^+,u\}$ and $\{v_r,v_r^-,v_r^+,u\}$ induce claws.
By Lemma \ref{le3} $(b)$ and the fact $G$ is 1-heavy, $v_2^-$, $v_r^+$ are heavy
or $v_2^+$ and $v_r^-$ are heavy vertices.

By Claim \ref{c2}, there exists a vertex $l_1\in C[v_1^{+},v_2^-)$ such that $v_2^-l_1\in E(G)$
and $l_1v_1\in E(G)$, and there exists a vertex ${l_1}'\in C(v_r^{+},v_1^-]$
such that $v_r^+{l_1}'\in E(G)$ and ${l_1}'v_1\in E(G)$.
\begin{claim}\label{c3}
$(1)$ $v_1,v_1^-,v_1^+,l_1^+,l_1^-,{l_1}'^-,{l_1}'^+$ are light vertices;\\
$(2)$ $v_1l_1^-\in E(G)$;\\
$(3)$ $l_1,{l_1}'$ are light vertices;\\
$(4)$ $v_1^-v_2\notin E(G)$;\\
$(5)$ $v_1^-l_1\in E(G)$ and $l_1v_2\in E(G)$.
\end{claim}

\begin{proof}
$(1)$ By Lemma \ref{le3} $(b)$, $(c)$ and $(e)$, it is obvious that 
$v_1,v_1^-,v_1^+,l_1^+,l_1^-,{l_1}'^-,{l_1}'^+$ are light vertices.

$(2)$ Suppose that $v_1l_1^-\notin E(G)$. Note that $v_1l_1^+\notin \widetilde{E}(G)$ 
and $l_1^-l_1^+\notin \widetilde{E}(G)$ by Lemma \ref{le3} ($d$). 
Since $l_1^+,v_1,l_1^-$ are light, $\{l_1,l_1^+,v_1,l_1^-\}$ induces a 
light claw, a contradiction.

$(3)$ Since $v_1l_1^-\in E(G)$, we have $v_2^+l_1\notin \widetilde{E}(G)$ and 
$v_r^+l_1\notin \widetilde{E}(G)$ by Lemma \ref{le3} $(e)$. Note that 
either $v_r^+$ or $v_2^+$ is a heavy vertex. This implies $l_1$ is 
a light vertex. The other assertion can be proved similarly.

$(4)$ Suppose that $v_1^-v_2\in E(G)$ and $v_2^-$, $v_3^+$ are heavy. 
Let $C'=v_1^-v_2C[v_2,v_r]v_ruv_1$ $C[v_1,v_2^-]v_2^-v_r^+C[v_r^+,v_1^-]$. 
Then $C'$ is an $o$-cycle such that $V(C)\subset V(C')$, a contradiction.

Suppose that $v_1^-v_2\in E(G)$ and $v_2^+$, $v_r^-$ are heavy. Now 
$\{v_2,v_2^-,u,v_1^-\}$ induces a light claw, a contradiction.

$(5)$ Suppose that $v_1^-l_1\notin E(G)$. Note that $uv_1^-\notin \widetilde{E}(G)$ 
by Lemma \ref{le3} ($a$) and $ul_1\notin \widetilde{E}(G)$ by Lemma \ref{le3} 
$(c)$ and $v_1l_1^-\in E(G)$. Now $\{v_1,l_1,u,v_1^-\}$ induces a light claw, 
a contradiction.

Suppose that $l_1v_2\notin E(G)$. Let $H=G[\{v_1^-,l_1,v_2^-,v_2,u\}]$. 
By Lemma \ref{le3}, we have $uv_2^-\notin E(G)$, $ul_1\notin E(G)$, 
$uv_1^-\notin E(G)$, $v_2v_1^-\notin E(G)$ and $v_1^-v_2^-\notin E(G)$. 
Now $G[\{v_1^-,l_1,v_2^-,v_2,u\}]$ is an induced path of length 4 in $G$. 
It follows that $d_H(v_1^-,u)=4$, contradicting the fact that $d_G(v_1^-,u)=2$ 
and $G$ is almost distance-hereditary. Hence we have $l_1v_2\in E(G)$.
\end{proof}

Now we consider the following two subcases.

\begin{subcase}
$v_2^-$, $v_r^+$ are heavy vertices.
\end{subcase}

By Lemma \ref{le3} $(c)$ and $v_1l_1^-\in E(G)$, $ul_1\notin \widetilde{E}(G)$. 
By Lemma \ref{le3} $(a)$ and $(e)$, we have $uv_2^+\notin \widetilde{E}(G)$ 
and $l_1v_2^+\notin \widetilde{E}(G)$. By Claim \ref{c3} $(1)$ and $(3)$, 
$\{v_2,l_1,u,v_2^+\}$ induces a light claw, a contradiction.

\begin{subcase}
$v_2^+$ and $v_r^-$ are heavy.
\end{subcase}

\begin{claim}\label{c4}
$G[\{v_r^+,{l_1}',l_1,v_2,u\}]$ is an induced path in $G$.
\end{claim}

\begin{proof}
It is easily to check that $v_r^+{l_1}'\in E(G)$, $l_1v_2\in E(G)$ (by Claim \ref{c3} (5))
and $v_2u\in E(G)$. By Lemma \ref{le3} $(a)$, $v_r^+u\notin E(G)$. From the
proof of Claim \ref{c3} $(3)$, we know that $v_r^+l_1\notin E(G)$.
Since $v_1l_1^-\in E(G)$, $ul_1\notin \widetilde{E}(G)$ by Lemma \ref{le3}
$(c)$. By symmetry, we have $u{l_1}'\notin \widetilde{E}(G)$ since $ul_1\notin \widetilde{E}(G)$.
Next we only need show that ${l_1}'l_1\in E(G)$, $v_r^+v_2\notin E(G)$ and ${l_1}'v_2\notin E(G)$.

\begin{subclaim}
${l_1}'l_1\in E(G)$.
\end{subclaim}

\begin{proof}
Otherwise, $\{v_1,l_1,{l_1}',u\}$ induces a light claw, contradicting $G$ is 1-heavy.
\end{proof}

\begin{subclaim}
$v_r^+v_2\notin E(G)$ and ${l_1}'v_2\notin E(G)$.
\end{subclaim}

\begin{proof}
Suppose that $v_r^+v_2\in E(G)$. Consider the subgraph induced by $\{v_2,v_2^-,v_r^+,u\}$. 
By Lemma \ref{le3} $(a)$, we have $uv_2^-\notin E(G)$ and $uv_r^+\notin E(G)$. Since 
$v_2^-,v_r^+,u$ are light and $G$ is 1-heavy, $v_r^+v_2^-\in E(G)$. Now 
$C'=v_1l_1^-\overleftarrow{C}[l_1^-,v_1^+]v_1^+v_1^-\overleftarrow{C}[v_1^-,v_r^+]
v_r^+v_2^-\overleftarrow{C}[v_2^-,l_1]l_1v_2C[v_2,v_r]\break v_ruv_1$ is an 
\emph{o}-cycle such that $V(C)\subset V(C')$, a contradiction.

Suppose that ${l_1}'v_2\in E(G)$. Consider the subgraph induced by 
$\{v_2,v_2^-,{l_1}',u\}$. Since $v_2^-,{l_1}',u$ are light and $G$ 
is 1-heavy, ${l_1}'v_2^-\in E(G)$. Now $C'=v_1uv_2C[v_2,{l_1}']{l_1}'v_2^-\break 
\overleftarrow{C}[v_2^-,v_1^+]v_1^+{l_1}'^+C[{l_1}'^+,v_1]$ is an 
\emph{o}-cycle such that $V(C)\subset V(C')$, a contradiction. (Note that
$v_1^+{l_1}'^+\in E(G)$ since otherwise $\{{l_1}',v_r^+,v_1^+,{l_1}'^+\}$
induces a light claw, a contradiction.)
\end{proof}
Now it is proved that $G[\{v_r^+,{l_1}',l_1,v_2,u\}]$ is an induced path in $G$.
\end{proof}

By Claim \ref{c4}, $G[\{v_r^+,{l_1}',l_1,v_2,u\}]$  is an induced path of length 4 in $G$. 
It follows that $d_H(v_r^+,u)=4$, contradicting the fact that $d_G(v_r^+,u)=2$ and 
$G$ is almost distance-hereditary.

\begin{case}
$v_2^-v_2\in E(G)$ or $v_r^-v_r^+\in E(G)$.
\end{case}
Without loss of generality, by symmetry, assume that $v_2^-v_2^+\in E(G)$.

By Claim \ref{c2}, for any vertex $v_i$ such that $v_i^-v_i^+\in E(G)$, 
there exists a vertex $l_i\in C[v_i^{+},v_{i+1}^-)$ such that $l_iv_i\in E(G)$
and $v_{i+1}^-l\in E(G)$.

\begin{claim}\label{c5}
$v_{i+1}^+l_i\in E(G)$.
\end{claim}

\begin{proof}
Suppose that $v_{i+1}^+l_i\notin E(G)$. Let $H=G[\{v_{i+1}^+,v_{i+1}^-,l_i,v_i,u\}]$. 
By Lemma \ref{le3} $(c)$, we have $v_iv_{i+1}^-\notin E(G)$ and $v_iv_{i+1}^+\notin E(G)$. 
We can see that $H$ is an induced path of length 4 in $G$. Hence $d_H(v_{i+1}^+,u)=4$, 
contradicting the fact that $d_G(v_{i+1}^+,u)=2$ and $G$ is almost distance-hereditary.
\end{proof}

By Claim \ref{c5} and Lemma \ref{le3} $(b)$, we have $l_i\neq v_i^+$.

\begin{claim}\label{c6}
$\{l_i,l_{i}^-,v_i,v_{i+1}^-\}$ induces a claw.
\end{claim}

\begin{proof}
By Lemma \ref{le3} $(c)$ and $(e)$, we have $v_iv_{i+1}^-\notin \widetilde{E}(G)$ and $l_{i}^-v_{i+1}^-\notin \widetilde{E}(G)$.
By Claim \ref{c5} and Lemma \ref{le3} $(e)$, we have $v_il_i^-\notin E(G)$. So $\{l_i,l_{i}^-,v_i,v_{i+1}^-\}$ induces a claw.
\end{proof}

\begin{subcase}
$v_r^-v_r^+\notin E(G)$.
\end{subcase}

\begin{claim}\label{c7}
$\{l_1,l_1^+,v_1,v_2^+\}$ induces a light claw.
\end{claim}

\begin{proof}
By Lemma \ref{le3} $(a)$, we have $uv_r^-\notin E(G)$ and $uv_r^+\notin E(G)$. Now $\{v_r,v_r^-,v_r^+,u\}$ 
induces a claw. Since $G$ is 1-heavy and $u$ is light, $v_r^-$ is heavy or $v_r^+$ is heavy.
\begin{subclaim}\label{s1}
$v_r^+$ is heavy.
\end{subclaim}

\begin{proof}
Suppose that $v_r^-$ is heavy. By Lemma \ref{le3} $(b)$, $(c)$ and $(e)$, $v_2^-$, $v_1$ and $l_1^-$ are light.
By Claim \ref{c6}, $\{l_1,l_1^-,v_1,v_2^-\}$ induces a light claw, a contradiction.
\end{proof}

\begin{subclaim}\label{s2}
$v_1^+,v_2^+$, $v_1$ and $l_1^+$ are light.
\end{subclaim}

\begin{proof}
By Claim \ref{s1}, $v_r^+$ is heavy. By Lemma \ref{le3} $(b)$, $(c)$ and $(e)$, $v_1^+,v_2^+$,
$v_1$ and $l_1^+$ are light.
\end{proof}

\begin{subclaim}\label{s3}
$v_1l_1^+\notin E(G)$.
\end{subclaim}

\begin{proof}
If $v_1l_1^+\in E(G)$, then consider the subgraph induced by $\{v_1,u,v_1^+,l_1^+\}$. 
It is obvious that $uv_1^+\notin \widetilde{E}(G)$ and $ul_1^+\notin \widetilde{E}(G)$. 
By Claim \ref{c5} and Lemma \ref{le3} $(f)$, we have $v_1^+l_1^+\notin E(G)$.
By Claim \ref{s2}, $\{v_1,u,v_1^+,l_1^+\}$ induces a light claw, a contradiction.
\end{proof}

Now consider the subgraph induced by $\{l_1,l_1^+,v_1,v_2^+\}$. By Claim \ref{c5} 
and Lemma \ref{le3} $(c)$, $(e)$, we have $v_1v_2^+\notin \widetilde{E}(G)$ and $l_1^+v_2^+\notin \widetilde{E}(G)$.
By Claim \ref{s3}, we have $v_1l_1^+\notin E(G)$. By Claim 7.2, $\{l_1,l_1^+,v_1,v_2^+\}$ induces a light claw.
\end{proof}
By Claim \ref{c7}, a contradiction.

\begin{subcase}
$v_r^-v_r^+\in E(G)$.
\end{subcase}

\begin{claim}\label{c8}
$v_1$ is heavy.
\end{claim}

\begin{proof}
By Claim \ref{c2}, there exists a vertex ${l_1}\in C[v_1^{+},v_2^-)$ such that $v_2^-l_1\in E(G)$
and $l_1v_1\in E(G)$. By Claims \ref{c5} and \ref{c6}, $v_2^+l_1\in E(G)$,
$\{l_r,l_r^-,v_r,v_1^-\}$ and $\{l_1,l_1^-,v_1,v_2^-\}$ induce claws.

Suppose that $v_2^-$ is heavy. By Lemma \ref{le3} $(b)$, $(c)$ and $(e)$, 
we have $v_1^-$, $v_r$ and $l_r^-$ are light. Now $\{l_r,l_r^-,v_r,v_1^-\}$ 
induces a light claw, contradicting that $G$ is 1-heavy.

Suppose that $l_1^-$ is heavy. By Claim \ref{c8}, Lemma \ref{le3} $(e)$ and $(f)$, since $v_2^+l_1\in E(G)$, 
we have $v_rl_1^-\notin \widetilde{E}(G)$ and $v_1^-l_1^-\notin \widetilde{E}(G)$. 
This implies that $v_r,v_1^-$ are light. At the same time, we can prove that ${l_r}^-$ 
is light (otherwise, 
$C'=v_ruv_2\overleftarrow{C}[v_2,l_1]l_1v_2^+C[v_2^+,v_r^-]v_r^-v_r^+C[v_r^+,l_r^-]l_r^-l_1^-\break \overleftarrow{C}[l_1^-,l_r]lv_r$ 
is an \emph{o}-cycle such that $V(C)\subset V(C')$, a contradiction). Now $\{l_r,l_r^-,v_r,v_1^-\}$ 
induces a light claw, contradicting that $G$ is 1-heavy.

Note that $\{l_1,l_1^-,v_1,v_2^-\}$ induces a claw and $v_2^-,l_1^-$ are light. 
We can see $v_1$ is heavy since $G$ is 1-heavy.
\end{proof}

By Claim \ref{c2}, there exists a vertex ${l_2}'\in C(v_1^{+},v_2^-]$ such 
that $v_1^+{l_2}'\in E(G)$ and ${l_2}'v_2\in E(G)$.

\begin{claim}\label{c9}
$l_r^-,l_r^+,{l_2}'^-,{l_2}'^+$ are light.
\end{claim}

\begin{proof}
If $l_r^+$ is heavy, then $C'=v_ruv_1l_r^+C[l_r^+,v_1^-]v_1^-v_1^+C[v_1^+,v_r^-]v_r^-v_r^+C[v_r^+,l_r]l_rv_r$ 
is an \emph{o}-cycle such that $V(C)\subset V(C')$, a contradiction. If $l_r^-$ is heavy, 
then $C'=v_rl_rC[l_r,v_1^-]v_1^-v_1^+\break C[v_1^+,v_r^-]v_r^-v_r^+C[v_r^+,l_r^-]l_r^-v_1uv_r$ 
is an \emph{o}-cycle such that $V(C)\subset V(C')$, a contradiction.

Similarly, by symmetry, we can prove that ${l_2}'^-,{l_2}'^+$ are light.
\end{proof}

\begin{claim}\label{c10}
$v_r$ is heavy.
\end{claim}

\begin{proof}
Consider the subgraph induced by $\{l_r,l_r^-,l_r^+,v_r\}$. By Claim \ref{c5} 
and Lemma \ref{le3} $(e)$, we have $v_rl_r^-\notin \widetilde{E}(G)$. 
By Lemma \ref{le3} $(d)$, we have $l_r^-l_r^+\notin \widetilde{E}(G)$ 
and $v_rl_r^+\notin \widetilde{E}(G)$. Since $G$ is 1-heavy, $v_r$ 
is heavy by Claim \ref{c9}.
\end{proof}

\begin{claim}\label{c11}
$\{v_1^-,v_1^+,l_r^+,{l_2}'^+\}$ induces a light claw.
\end{claim}

\begin{proof}
By Claim \ref{c9}, $l_r^+,{l_2}'^+$ are light. Note that $v_1^-v_1^+\in E(G)$. Now we suffice 
to check the following facts: $\{v_1^-l_r^+, v_1^-{l_2}'^+\}\subset E(G)$ and 
$\{v_1^+l_r^+,v_1^+{l_2}'^+,l_r^+{l_2}'^+\}\cap E(G)=\emptyset$.

\begin{subclaim}
$v_1^-l_r^+\in E(G)$ and $v_1^-{l_2}'^+\in E(G)$.
\end{subclaim}

\begin{proof}
Suppose that $v_1^-l_r^+\notin E(G)$. By Lemma \ref{le3} $(d)$ and $(e)$, we have $l_r^-l_r^+\notin \widetilde{E}(G)$ 
and $v_1^-l_r^-\notin \widetilde{E}(G)$. Now $\{l_r,l_r^-,l_r^+,v_1^-\}$ induces a light claw, a contradiction.

By Claim \ref{c5} and symmetry, we have $v_1^-{l_2}'\in E(G)$. Suppose that $v_1^-{l_2}'^+\notin E(G)$. By Lemma \ref{le3} 
$(d)$ and $(e)$, we have ${l_2}'^-{l_2}'^+\notin \widetilde{E}(G)$ and $v_1^-{l_2}'^-\notin \widetilde{E}(G)$.
 Now $\{{l_2}',{l_2}'^-,{l_2}'^+,v_1^-\}$ induces a light claw, a contradiction.
\end{proof}

\begin{subclaim}
$\{v_1^+l_r^+,v_1^+{l_2}'^+,l_r^+{l_2}'^+\}\cap E(G)=\emptyset$.
\end{subclaim}

\begin{proof}
By Lemma \ref{le3} $(e)$, since $v_rl_r\in E(G)$ and $v_2{l_2}'\in E(G)$, we have 
$v_1^+l_r^+\notin \widetilde{E}(G)$ and $v_1^+{l_2}'^+\notin \widetilde{E}(G)$. 
At the same time, we can prove that $l_r^+{l_2}'^+\notin E(G)$, since otherwise, $C'=v_ruv_2{l_2}'\overleftarrow{C}[{l_2}',{l_r}^+]{l_r}^+{l_2}'^+C[{l_2}'^+,v_2^-]v_2^-v_2^+C[v_2^+,v_r^-]v_r^-v_r^+C[v_r^+,l_r]l_rv_r$ 
is an \emph{o}-cycle such that $V(C)\subset V(C')$, a contradiction.
\end{proof}
It is proved that $\{v_1^-,v_1^+,l_r^+,{l_2}'^+\}$ induces a light claw.
\end{proof}
By Claim \ref{c11}, a contradiction. The proof of Theorem 4 is complete. {\hfill$\Box$}


\end{document}